\documentclass[10pt]{article}
\usepackage{graphicx}
\usepackage{amsfonts,amssymb,amscd,amsmath,enumerate,verbatim,calc}
\begin{document}
\makeatletter

\def\endofproofmark{$\Box$}

\def\RR{I\!\!R}
\def\NN{I\!\!N}

\def\endofproofmark{$\Box$}

 \title{The Fractional Local Metric Dimension of Graphs
            }

\author{ Hira Benish, Muhammad Murtaza, Imran Javaid*\\
{\it Centre for Advanced Studies in Pure and Applied Mathematics,}\\
{\it Bahauddin Zakariya University Multan, Pakistan}\\
Email:  {\tt hira\_benish@yahoo.com}, {\tt mahru830@gmail.com},\\
{\tt imran.javaid@bzu.edu.pk}.
 }

\date{}
\maketitle

\pagestyle{myheadings}
\newtheorem{Theorem}{Theorem}[section]
\newtheorem{Lemma}[Theorem]{Lemma}
\newtheorem{Corollary}[Theorem]{Corollary}
\newtheorem{Proposition}[Theorem]{Proposition}
\newtheorem{Observation} [Theorem] {Observation}
\newtheorem{Remark}[Theorem]{Remark}
\newtheorem{Remarks}[Theorem]{Remarks}
\newtheorem{Example}[Theorem]{Example}
\newtheorem{Examples}[Theorem]{Examples}
\newtheorem{Definition}[Theorem]{Definition}
\newtheorem{Problem}[Theorem]{}
\newtheorem{Conjecture}[Theorem]{Conjecture}
\newtheorem{Question}[Theorem]{Question}

\def\frameqed{\framebox(5.2,6.2){}}
\def\deshqed{\dashbox{2.71}(3.5,9.00){}}
\def\ruleqed{\rule{5.25\unitlength}{9.75\unitlength}}
\def\myqed{\rule{8.00\unitlength}{12.00\unitlength}}
\def\qed{\hbox{\hskip 6pt\vrule width 7pt height11pt depth1pt\hskip 3pt}
\bigskip}
\newenvironment{proof}{\trivlist\item[\hskip\labelsep{\bf Proof}:]}{\hfill
 $\frameqed$ \endtrivlist}
\newcommand{\COM}[2]{{#1\choose#2}}

\textwidth=11.4cm \textheight=18.1cm \thispagestyle{empty} \null

\begin{abstract}
The fractional versions of graph-theoretic invariants multiply the range of applications in scheduling, assignment and operational research problems.
In this paper, we introduce the fractional version of local metric dimension of graphs. The local resolving neighborhood $L(xy)$ of an edge $xy$ of a graph $G$ is the set of those vertices in $G$ which
resolve the vertices $x$ and $y$. A function $f:V(G)\rightarrow[0, 1]$ is a local resolving function of $G$ if $f(L(xy))\geq1$ for all edges $xy$ in $G$. The minimum value of $f(V(G))$ among all
local resolving functions $f$ of $G$ is the fractional local metric
dimension of $G$. We study the properties and bounds of fractional local metric dimension of graphs and give some characterization results. We determine the fractional local metric dimension of strong and cartesian product of graphs.
\end{abstract}

\section{Introduction and Terminology}
Resolving sets and metric dimension of a graph were introduced by Slater \cite{9} and Harary and Melter \cite{har} independently. Currie {\it et al.} \cite{CO} initiated the concept of fractional metric dimension and defined it as the optimal solution of linear relaxation of the integer programming problem of the metric dimension of graphs. The fractional metric dimension problem was further studied by Arumugam and Mathew \cite{AM} in 2012. The authors provided a sufficient condition for a connected graph $G$ whose fractional metric dimension is $\frac{|V(G)|}{2}$. The fractional metric dimension of graphs and graph products has also been studied in \cite{AM, FLW, FW, FW2, KS, EY}.

Okamoto {\it et al.} \cite{10} initiated the study of distinguishing adjacent vertices in a graph $G$ rather than all the vertices of $G$ by distance. This motivated the study of local resolving sets and local metric dimension in graphs. In this paper, we introduce the fractional version of local metric dimension of graph. We study the local fractional metric dimension of some graphs and establish some bounds on the fractional local metric dimension
of graphs. We also determine the fractional local metric dimension of strong and cartesian product of graphs.

Let $G=(V(G),E(G))$ be a finite, simple and connected graph with $V(G)$ and $E(G)$ be the vertex set and the edge set of $G$, respectively. The edge
between two vertices $u$ and $v$ is denoted by $uv$. If two vertices $u$ and $v$ are joined by an edge then they are
called adjacent vertices, denoted by $u\sim v$. $N_G(u)=\{v \in
V(G):vu\in E(G)\}$ and $N_G [u] = N (u) \cup\{u\}$ are called the {\it open neighborhoods}
and the {\it closed neighborhoods} of a vertex $u$, respectively. For a subset
$U$ of $V (G)$, $N_G (U ) = \{v \in V (G) :uv\in E(G);u\in
U\}$ is the open neighborhood of $U$ in $G$. The {\it distance} between
any two vertices $u$ and $v$ of $G$ is the shortest length of a path
between $u$ and $v$, denoted by $d(u,v)$. Two distinct vertices
$u,v$ are {\it adjacent twins} if $N[u]=N[v]$ and {\it non-adjacent
twins} if $N(u)=N(v)$. Adjacent twins are called {\it true twins}
and non-adjacent twins are called {\it false twins}. For two distinct vertices $u$ and $v$ in $G$, $R(u,v)=\{x\in V(G):
d(x,u)\ne d(x,v)\}$. A vertex set $W\subseteq V(G)$ is called a {\it resolving set} of $G$ if $W\cap
R(u,v)\neq\emptyset$ for any two distinct vertices $u,v\in V(G)$. The
minimum cardinality of a resolving set of $G$ is called the {\it
metric dimension} of $G$. Let $f$ be a function such that $f$ assigns a number
between 0 and 1 to each vertex of $G$ i.e.,
$f:V(G)\rightarrow[0,1]$. The function $f$ is called a {\it
resolving function} of $G$ if $f(R(u,v))\geq 1$ for any two distinct
vertices $u$ and $v$ in $G$. The minimum value of $f(V(G))$ among all
resolving functions $f$ of $G$ is called the {\it
fractional metric dimension} of $G$, denoted by $dim_f(G)$.

A vertex set $W \subset V(G)$ is called a {\it local
resolving set} of $G$ if $W \cap R(u,v)\neq \emptyset$ for any two adjacent vertices $u,v\in V(G)$. The minimum cardinality of
a local resolving set is called the {\it local metric dimension} of
$G$ and it is denoted by $ldim(G)$. A local resolving set of order
$ldim(G)$ is called a {\it local metric basis} of $G$. For $uv\in E(G)$, we define {\it the local resolving neighborhood}
as $L(uv)=\{x\in V(G);d(u,x)\neq d(v,x)\}$. $L(uv)=V(G)$, for all
$uv\in E(G)$, if and only if $ldim(G)=1$. In \cite{10}, it was shown
that $ldim(G)=1$ if and only if $G$ is a bipartite graph. Hence,
$L(uv)=V(G)$ for all $uv\in E(G)$ if and only if $G$ is a bipartite
graph. Now, we define the fractional
local metric dimension of a graph as follows;
\begin{Definition}
A function $f:V\rightarrow [0,1]$ is a local resolving function
$LRF$ of $G$ if $f(L(uv))\geq1$ for all $uv\in E(G)$, where
$f(L(uv))=\sum\limits_{x\in L(uv)}f(x)$. The weight of local
resolving function $f$ is defined as $|f|=\sum\limits_{v\in
V(G)}f(v)$. The minimum weight of a local resolving function of $G$ is
called the fractional local metric dimension $ldim_f(G)$ of $G$.
\end{Definition}

This paper is organized as follows: in Section 2, we characterize the
graphs $G$ with the fractional local metric dimension $\frac{|V(G)|}{2}$ and give bounds on the fractional local metric dimension of graphs. We study the fractional local metric dimension of some families of graphs and also discuss the differences between the fractional metric dimension and the fractional local metric dimension of some families of graph. In Section 3, we study the fractional local metric dimension of strong and cartesian product of graphs. We establish bounds on the fractional local metric dimension of these graph products.



\section{Characterization Results and Bounds on $ldim_f(G)$}
In a connected graph $G$, since every resolving set is a local resolving set, therefore every resolving function is also a local
resolving function but every local resolving function is not a resolving function. Thus
$$ldim_f(G)\leq dim_f(G)$$
Since, the characteristic function of a minimal local resolving set is an $LRF$ of $G$, therefore
$$1\leq ldim_f(G)\leq
ldim(G)\leq n-1.$$
Thus, if a graph $G$ has $ldim(G)=1$, then $ldim_f(G)=1$. We have the following result:
\begin{Lemma}\label{b1}
Let $G$ be a graph of order $n\geq 2$, then
$ldim_{f}(G)=1$ if and only if $G$ is a bipartite graph.
\end{Lemma}
\begin{proof}
The sufficiency is immediate using the bounds given above. Conversely,
let $G$ be a graph with $ldim_f(G)=1$, then $|L(uv)|=n$ for all
$uv\in E(G)$. Suppose $G$ is not a bipartite graph and $G$ contains
an odd cycle $C_s=\{v_i: v_i\sim v_{i+1}, 1\le i \le s, v_{s+1}=v_1 \}$, where $s\leq n$ is odd. Note that
$|L(v_iv_{i+1})|=s-1$ for all $v_iv_{i+1}\in E(C_s)$, $i\in
\{1,2,...,s\}$, a contradiction. Hence, $G$ is a bipartite graph.
\end{proof}
Although there is a striking difference between fractional metric dimension and fractional local metric dimension of graphs, however the same results hold for the local metric dimension of graph when graph has true twin vertices. Let $G$ be a graph and $uv\in E(G)$, then $d(u, x) = d(v, x)$ for all
$x\in V(G)-\{u,v\}$ if and only if $u$ and $v$ are true twins. We have the following result about the local resolving neighborhood of true twin vertices:
\begin{Lemma}\label{twin}
Let $G$ be a graph and $uv\in E(G)$. Then
$\{u,v\}\subseteq L(uv)$. Moreover, we have $L(uv)=\{u,v\}$ if and
only if $u$ and $v$ are true twins.
\end{Lemma}
\begin{proof}
The proof simply follows from the fact that $L(uv)=\{u,v\}$ if and only if $d(u,x)=d(v,x)$ for all $x\in V(G)\setminus \{u,v\}$.
\end{proof}
Given a graph $H$ and a family of graphs $\mathcal{I}=\{I_v\}_{v\in
V(H)}$, indexed by $V(H)$, their {\em generalized lexicographic
product}, denoted by $H[\mathcal{I}]$, is defined as the graph with
the vertex set $ V(H[\mathcal{I}])=\{(v,w):v\in V(H)\textup{ and }
w\in V(I_v)\} $ and the edge set $ E(H[\mathcal{I}])
=\{(v_1,w_1)(v_2,w_2):v_1v_2\in E(H), \textup{or }v_1=
v_2\textup{ and }w_1w_2\in E(I_{v_1})\}.$
We state the result
as follows:
\begin{Theorem}\label{mthm}
Let $G$ be a connected graph of order $n\geq 2$. Then the following
statements are pairwise equivalent.

  {\rm(i)} $ldim_f(G)=\frac{n}{2}$.

  {\rm(ii)} Each vertex in $G$ has a true twin.

  {\rm(iii)} There exist a graph $H$ and a family of graphs $\mathcal{I}=\{I_v\}_{v\in
V(H)}$, where $I_v$ is a non-trivial complete graph, such that $G$ is
isomorphic to $H[\mathcal I]$.
\end{Theorem}
\begin{proof}
(i) $\Rightarrow$ (ii) Suppose (i) holds. If there exists a vertex $u$ in $G$ such that $u$ does not have a true twin, then
the following function $f: V(G)\rightarrow[0,1]$,
$$f(x)=\left\{\begin{array}{ll}
                                       0, & \textup{if }x=u, \\
                                       \frac{1}{2}, & \textup{if }x\ne u,
                                     \end{array}\right.
$$
is a local resolving function of $G$ by Lemma \ref{twin}, which implies that
$ldim_f(G)\le\frac{n-1}{2}$, a contradiction.\\
(ii) $\Rightarrow$ (iii)
Suppose (ii) holds. For   $x,y\in V(G)$, define $x\equiv y$ if and
only if $x=y$ or $x, y$ are true twins. It is clear that $\equiv$ is an
equivalence relation. Suppose $O_{1}, \ldots, O_{m}$ where $m\le n$ the equivalence classes. Then the induced subgraph on each
$O_i$, where $i\in{1,...,m}$ denoted also by $I_{O_i}$, is a non-trivial null graph  or a
non-trivial complete graph. Let $H$ be the graph with the vertex set
$\{O_1,\ldots,O_m\}$, where two distinct vertices $O_i$ and $O_j$
are adjacent if  there exist $x\in O_i$ and $y\in O_j$ such that $x
$ and $y$ are adjacent in $G$. It is routine to verify that $G$ is
isomorphic to $H[\mathcal I]$, where $\mathcal
I=\{I_{O_i}:i=1,\ldots,m\}$.\\
(iii) $\Rightarrow$ (i) Suppose (iii) holds. For $v\in V(H)$, write
$$
V(I_v)=\{w_v^1,\ldots,w_v^{s(v)}\}.
$$
Where $|I_v|=s(v)$. Then $s(v)\ge 2$, and $(v,w_v^i)$ and $(v,w_v^j)$ are true twins in
$H[\mathcal I]$, where $1\le i<j\le s(v)$. Let $h$ be a local resolving
function of $H[\mathcal I]$ with $|h|=ldim_f(H[\mathcal I])$. By
Lemma \ref{twin}, we get
\begin{equation*}
  h((v,w_v^i))+h((v,w_v^j))\ge 1\quad\textup{for }1\le i<j\le s(v),
\end{equation*}
which implies that $\sum_{k=1}^{s(v)}h(v,w_v^k)\ge\frac{s(v)}{2}$, and so
$ldim_f(G)=ldim_f(H[\mathcal I])=|h|=\sum_{v\in V(H)}\sum_{k=1}^{s(v)}h((v,w_v^k))
 \ge \sum_{v\in V(H)}\frac{s(v)}{2} =\frac{|V(H[\mathcal I])|}{2} =\frac{n}{2}.$
\end{proof}
The join graph $G_1+G_2$ is the graph obtained from $G_1$ and $G_2$ by
joining each vertex of $G$ with every vertex of $H$. Note that, if
each vertex in $G_i$ has a true twin for $i\in\{1,2\}$, then each
vertex in $G_1+G_2$ has a true twin. We have the following result:
\begin{Corollary}\label{c2}
Let $\Theta$ denotes the collection of all connected graphs $G$ with
$ldim_f(G)=\frac{|V(G)|}{2}$. If $G_1,G_2\in \Theta$, then
$G_1+G_2\in \Theta$.
\end{Corollary}
The next result is a generalization of Theorem \ref{mthm}. The
clique of a graph $G$ is a complete subgraph in $G$.
\begin{Theorem}\label{t2}
Let $G$ be a connected graph of order $n$ and $W_1,W_2,....,W_k$ be independent cliques in $G$ with $|W_i|\geq3$ for all $i$, $(1\leq i\leq k)$. Then
$ldim_f(G)=\sum\limits_{i=1}^k\frac{|V(W_i)|}{2}$ if and only if for all $uv\in E(G)\setminus E(W_i)$, $L(xy)\subseteq L(uv)$ for some $xy\in E(W_i)$ for some $i$, $(1\leq i\leq k)$.
\end{Theorem}
\begin{proof}
Let $G$ be a graph with $ldim_f(G)=\sum\limits_{i=1}^k\frac{|V(W_i)|}{2}$, then there exists a local resolving function $f:V(G)\rightarrow [0,1]$ defined as:
$$f(v) =\left\{
  \begin{array}{ll}
    \frac{1}{2}\,\,\,\     & \mbox{if}\,\,\,v\in V(W_i), 1\leq i \leq k ,\\
      0        \,\,\,\     & \mbox{otherwise}.
  \end{array}
\right.
$$
$f(L(uv))\geq 1$ for all $uv\in E(G)\setminus E(W_i)$, for all $i$, $1\leq i \leq k$ is possible only when $L(xy)\subseteq L(uv)$ for some $xy\in E(W_i)$, for some $i$, $(1\leq i \leq k)$, since $f$ assigns 0 to the vertices of $V(G)\setminus V(W_i)$ for all $i$, $1\leq i \leq k$.\\
Conversely, suppose that for all $uv\in E(G)\setminus E(W_i)$, $L(xy)\subseteq L(uv)$ for some $xy\in E(W_i)$, for some $i$, $(1\leq i\leq k)$. Let $f:V(G)\rightarrow [0,1]$ be the function defined as:
$$f(v) =\left\{
  \begin{array}{ll}
    1/2\,\,\,\     & \mbox{if}\,\,\,v\in V(W_i), 1\leq i \leq k ,\\
      0        \,\,\,\     & \mbox{otherwise}.
  \end{array}
\right.
$$
It is clear that $f(L(uv))\geq1$ for all $uv\in E(G)$, since $L(xy)\subseteq L(uv)$. Hence $f$ is a local resolving function of $G$ and $ldim_f(G)\leq \sum\limits_{i=1}^k\frac{|V(W_i)|}{2}$. To show that $\sum\limits_{i=1}^k\frac{|V(W_i)|}{2}\leq ldim_f(G)$, suppose that $f$ is local resolving function of $W_i$ and not a local resolving function of $G$. Then there exist $uv\in E(G)$ such that $f(L(uv))<1$. This leads to a contradiction to our supposition that $L(xy)\subseteq L(uv)$. Hence, $ldim_f(G)= \sum\limits_{i=1}^k\frac{|V(W_i)|}{2}$.
\end{proof}
A lollipop graph $L_{m,n}$ is a graph obtained by joining a complete graph $K_m$ to a path $P_n$ with an edge.
\begin{Corollary}
Let $L_{m,n}$ be a lollipop graph with $m\geq 3$ and $n\geq 2$. Then $ldim_f(L_{m,n})=\frac{m}{2}$.
\end{Corollary}
\begin{proof}
Since for all $uv\in E(P_n)$, $L(xy)\subseteq L(uv)$ for some $xy\in E(K_m)$, hence by Theorem \ref{t2} and Theorem \ref{mthm}, $ldim_f(L_{m,n})=\frac{m}{2}$.
\end{proof}
Let $G$ be a graph of order $n$, we define
$l(G)=\min\{|L(uv)|: uv\in E(G)\}.$
\begin{Remark}\label{rg}
Let $r(G)=\min\{|R(u,v)|: u,v\in V(G)\}$ as defined in \cite{FLW}.
Note that for any graph $G$, $r(G)\leq l(G)$.
\end{Remark}
In the following result, we express the fractional local metric dimension of $G$ in
terms of $l(G)$.
\begin{Proposition}\label{1}
Let $G$ be a graph, then $ldim_{f}(G)\leq
\frac{|V(G)|}{l(G)}.$
\end{Proposition}
\begin{proof}
Let $f:V(G)\rightarrow[0,1]$, defined by
$f(x)= \frac{1}{l(G)}.$ For any two adjacent vertices $x$ and
$y,$ we have $f(L(xy))=\frac{|L(xy)|}{l(G)}\geq1$. Clearly, $f$ is
a local resolving function of $G$. Hence, $ldim_{f}(G)\leq
|f|=\frac{|V(G)|}{l(G)}$.
\end{proof}
By Lemma \ref{twin}, $\{u,v\}\in L(u,v)$ so it is clear that $|L(uv)|\geq 2$ for all $uv\in E(G)$.
We have the following corollary of Proposition \ref{1}.
\begin{Corollary}\label{cro1n}
For a graph $G$ of order $n$, $ldim_f(G)\leq \frac{n}{2}$.
\end{Corollary}
\begin{Lemma}\label{al}
Let $G$ be a graph and $U$ be a subset of $V(G)$ with cardinality
$|V(G)|-ldim(G)+1,$ there exists an edge $xy\in E(G)$ such that
$L(xy)\subseteq U.$
\end{Lemma}
\begin{proof}
Suppose there exists a subset $U$ with cardinality $|V(G)|-ldim(G)+1$ such
that $L(xy)\nsubseteq U,$ for all $xy\in E(G)$. Then $L(xy)\cap \{V(G)\backslash U\}\neq \emptyset$. So
$V(G)\backslash U$ is a local resolving set of $G.$ Therefore,
$ldim(G)-1=|V(G)\backslash U|< ldim(G),$ a contradiction.
\end{proof}

\begin{Theorem}\label{a2}
Let $G$ be a graph. Then $l(G)=|V(G)|-1$ if and only if $G$ is
isomorphic to an odd cycle.
\end{Theorem}
\begin{proof}
It is easy to verify that $l(G)=|V(G)|-1$ when $G$ is an odd cycle. Conversely, let $G$ be a graph of order $n\geq 4$ and $l(G)=|V(G)|-1$. We further suppose that $G$ is not a bipartite graph, since $l(G)=n$ for a bipartite graph of order $n$. Thus $G$ contains an odd cycle. Let $C_p:x_1,x_2,...,x_p$ are the vertices of odd cycle, where $p\leq n$ is odd. Let $\Delta(G)$ be the maximum degree of $G$. We claim that $\Delta(G)=2$. Suppose to the contrary that $\Delta\geq3$, then odd cycle $C_p$ must be a proper subgraph of $G$. Since $G$ is connected, therefore there exists a vertex $y\in V(G)\setminus V(C_p)$ such that $y$ is adjacent to any vertex, say $x_p$ of $C_p$. Since $C_p$ is an odd cycle, therefore $d(x_p,x_{\frac{p-1}{2}})=d(x_p,x_{\frac{p+1}{2}})$. Thus $x_p,y\notin L(x_{\frac{p-1}{2}}x_{\frac{p+1}{2}})$. Hence $|L(x_{\frac{p-1}{2}}x_{\frac{p+1}{2}})|\leq n-2$ which is a contradiction. Hence $\Delta(G)=2$ and $G$ is isomorphic to an odd cycle.
\end{proof}
Using Lemma \ref{a2}, we have the following result:
\begin{Theorem}
Let $G$ be a graph of order $n$. Then $$ldim_f (G)\geq \frac{ n}{
n-ldim(G) + 1}.$$
\end{Theorem}
\begin{proof}
Write $s= n-ldim(G) + 1$. Suppose $f$ is a local resolving function
of $G$ with $|f|=ldim_f(G)$. Let $\tau=\{T: T\subset V(G),|T|=n-ldim(G) + 1\}$ and $|\tau|={|V(G)|\choose s}$. For
each $U\in \tau$, $f (U)\geq1$ by Lemma \ref{al}. Hence, $\sum\limits_{U\in \tau} f(U)\geq {n\choose s}.$ Since $\sum\limits_{U\in \tau}f(U)={{n-1}\choose {s-1}}|f|$, so we accomplish our
result.
\end{proof}
\begin{Theorem}
For every integer $\epsilon,\delta$, there exist graphs $G$ and $H$
such that $dim_f(G)-ldim_f(G)\geq \delta$ and $dim_f(H)-ldim_f(H)\leq
\epsilon$.
\end{Theorem}
\begin{proof}
For the first inequality, we consider complete bipartite graph $K_{n,n}$, for which $dim_f(K_{n,n})=n$ \cite{AM} and $ldim_f(K_{n,n})=1$.
The difference between fractional metric dimension and fractional local metric dimension is $n-1>\delta$,
where $\delta$ can be as large as we like. For the second inequality, we consider cyclic graph $C_n$ of even order for which $dim_f(C_n)=\frac{n}{n-2}$ \cite{AM}, and $ldim_f(C_n)=1$. The
difference between fractional metric dimension and fractional local
metric dimension is $\frac{2}{n-2}<\epsilon$, where $\epsilon$ can
be as small as we like.
\end{proof}
Let $G$ be the complete $k$-partite graph $K_{a_1,a_2,...,a_k}$, for
$k>2$, of order $n=\sum\limits_{i=1}^ka_i$. Let $V(G)$ be
partitioned into $k$-partite sets $V_1,V_2,...,V_k$, where
$|V_i|=a_i$ for $1\leq i\leq k$. Okamoto et al. proved that
$ldim(K_{a_1,a_2,...,a_k})=k-1$ \cite{10}.
\begin{Lemma}
Let $G$ be the complete $k$-partite graph $K_{a_1,a_2,...,a_k}$, for
$k>2$, of order $n=\sum\limits_{i=1}^ka_i$. Then
$ldim_f(K_{a_1,a_2,...,a_k})=k-1$.
\end{Lemma}
\begin{proof}
Firstly, we show that $ldim_f(G)\leq k-1$. It is clear that all
$xy\in E(K_{a_1,a_2,...,a_k})$ if and only if $x\in V_i$ and
$y\in V_j$, $i\neq j$ and $i,j\in \{1,2,...,k\}$. Note that for all $xy\in E(K_{a_1,a_2,...,a_k})$,
$L(xy)=V_i\cup V_j$. One of the possible choices of local resolving function $f$ of $G$ is that $f$ is defined as: $f$ assigns 1 to only one vertex of $V_i\cup V_j$ and 0 to all other vertices of $V_i\cup V_j$. This implies $f(L(xy))\ge 1$ for all $xy\in E(G)$ and $|f|=k-1$. Thus $ldim_f(G)\leq k-1$.

To prove $k-1\leq ldim_f(G)$, suppose on contrary that $ldim_f(G)= k-2$. The minimum weight $k-2$ of a function $f$ among all the local resolving functions of $G$ will be possible only when $f$ assigns 0 to all vertices of $V_r\cup V_s$, for some $r,s\in \{1,2,...,k\}$. This implies $f(L(xy))<1$ for $xy\in E(G)$ where $x\in V_r$ and $y\in V_s$, which is a contradiction. Hence $ldim_f(G)=k-1$.
\end{proof}
The {\it automorphism group} of a graph $G$ is the set of all permutations
of the vertex set of $G$ that preserve adjacencies and non-adjacencies
of vertices in $G$ and it is denoted by $\Gamma(G)$. A graph $G$ is {\it vertex-transitive} if its automorphism group
$\Gamma(G)$ acts transitively on the vertex set. The {\it stabilizer} of a vertex $v\in V(G)$, denoted by $\Gamma_v$, is defined as
$\Gamma_v=\{\pi\in\Gamma:\pi(v)=v\}$. The {\it index} of a subgroup is defined as the number of distinct cosets of the subgroup in that group. In a vertex-transitive graph $G$, for any two vertices $v$ and $w$ in $V(G)$, $\Gamma_v$ and
$\Gamma_w$ are isomorphic and the index of $\Gamma_v$ in $\Gamma(G)$ is
equal to the order of $V(G)$. For a vertex-transitive graphs, if $l(G)=r(G)$, then
$ldim_f(G)=dim_f(G)$. For example, an odd cycle of order $n$ is a
vertex-transitive graph and $l(C_n)=r(C_n)$ for odd $n$. Petersen graph is a vertex-transitive graph and
$l(G)=6=r(G).$ Therefore, $ldim_f(P)=\frac {5}{3}=dim_f(G).$ But in general, $ldim_f(G)\neq dim_f(G)$ for vertex-transitive graphs. For instance, hypercube $Q_n$ is a vertex-transitive graph with $dim_f(Q_n)=2\neq 1 = ldim_f(Q_n)$. In the following result, we give the
fractional local metric dimension of a vertex-transitive graph $G$ in terms of the parameter $l(G)$.
\begin{Theorem}\label{2}
Let $G$ be a vertex-transitive graph. Then
$ldim_{f}(G)=\frac{|V(G)|}{l(G)}.$
\end{Theorem}
\begin{proof}
Let $l(G)=p$, then there exists an edge
$uv\in E(G)$ such that $|L(uv)|=p$. Suppose
$L(uv)=\{r_1,r_2...,r_p\}$. Let $\alpha \in \Gamma(G)$,
$L(\alpha(u)\alpha(v))=\{\alpha(r_1),\alpha(r_2),...,\alpha(r_p)\}$.
Let $f$ be a local resolving function of $G$ with $ldim_f(G)=|f|$. Then
$$f(\alpha(r_1))+f(\alpha(r_2))+...+f(\alpha(r_p))=f(L(\alpha(u)\alpha(v)))\geq
1,$$ which implies that
$$\sum\limits_{\alpha\in
\Gamma(G)}(f(\alpha(r_1))+f(\alpha(r_2))+...+f(\alpha(r_p)))\geq
|\Gamma(G)|.$$ Since $G$ is vertex-transitive, we have
$$|\Gamma_{r_1}|.|f|+|\Gamma_{r_2}|.|f|+...+ |\Gamma_{r_p}|.|f |\geq
|\Gamma(G)|$$ which implies that $ldim_f(G)\geq \frac{|V(G)|}{p}$. By
Proposition \ref{1}, we have the required result.

\end{proof}
Let $G$ be a connected graph, for $v\in V(G)$, $G-v$ is known as the vertex deletion subgraph of
$G$ obtained by deleting $v$ from the vertex set of $G$ along with
its incident edges.
\begin{Proposition}\label{propcut}
Let $G$ be a graph and $v\in V(G)$, then $ldim_f(G)-1\leq ldim_f(G-v)$.
\end{Proposition}
\begin{proof}
Let $f:V(G-v)\rightarrow [0,1]$ be a local resolving function of $G-v$ such
that $ldim_f(G-v)=|f|$. Consider a function $f^{\prime}:V(G)\rightarrow
[0,1]$ defined as:
$$f^{\prime}(u) =\left\{
  \begin{array}{ll}
    f(u),   \,\,\,\     & \mbox{if}\,\,\,u\neq v,\\
    1,   \,\,\,\     & \mbox{if}\,\,\,u=v.
  \end{array}
\right.$$ is a local resolving function of $G$ and $ldim_f(G)\leq
|f^{\prime}|$. Thus $ldim_f(G-v)=|f|=|f^{\prime}|-1\geq ldim_f(G)-1$.
\end{proof}
The fan graph $F_{1,n}$ of order $n + 1$ is defined as the join
graph $K_1 + P_n.$ Let $V(K_1)=\{u\}$ and
$V(P_n)=\{u_1,u_2,...,u_n\}$.
\begin{Lemma}
Let $F_{1,n}$ be a fan graph with $n\geq 3$, then
$$ldim_f(F_{1,n}) =\left\{
  \begin{array}{ll}
    2,   \,\,\,\     & \mbox{if}\,\,\,n=3, \\
    \frac{n}{3},   \,\,\,\     & \mbox{if}\,\,\,n\geq4.
  \end{array}
\right. $$
\end{Lemma}
\begin{proof}
Since $l(F_{1,3})=2$, therefore $ldim_f(F_{1,3})\leq 2$ by Proposition \ref{1}. Now, we show that
$2\leq ldim_f(F_{1,3})$. Since $l(F_{1,3})=2$ and $|L(xy)|\neq 4$ for any $xy\in E(F_{1,3})$. Thus a function $f:V(F_{1,3}) \rightarrow [0,1]$ is a local resolving function for $F_{1,3}$ if it assign 1/2 to each vertex of $F_{1,3}$. Otherwise there exists an edge $xy\in E(F_{1,3})$ such that $L(xy)< 1$. Hence $ldim_f(F_{1,3})=2$.\\
Let $F_{1,n}$ be a fan graph with $n\geq4$. Note that $\{u\}=V(K_1)$ does not locally resolve any $xy\in E(F_{1,n})$ for $x,y\neq u$. Let $f;V(F_{1,n})\rightarrow [0,1]$ is a local resolving function defined as:
$$f(v) =\left\{
  \begin{array}{ll}
    1/3,   \,\,\,\     & \mbox{if}\,\,\,v\neq u, \\
    0,    \,\,\,\     & \mbox{if}\,\,\,v=u.
  \end{array}
\right. $$
$f(L(xy))\geq 1$ for all $xy\in E(F_{1,n})$. Thus $|f|=\frac{n}{3}$. Hence $ldim_f(F_{1,n})\leq \frac{n}{3}$.\\
Now we show that $\frac{n}{3}\leq ldim_f(F_{1,n})$. Note that $l(F_{1,n})=3$ for $n\geq4$. $f$ is a local resolving function as defined above. If $f$ assigns 0 to any vertex from $V(P_n)$, then there exists an edge $xy\in E(F_{1,n})$ such that $f(L(xy))< 1$. Hence $ldim_f(F_{1,n})=\frac{n}{3}$ for $n\geq4$.
\end{proof}
\section{The Fractional Local Metric Dimension of Strong and Cartesian Product of Graphs}
The strong product of two graphs $G$ and $H$, denoted by $G\boxtimes
H$, is a graph with the vertex set $V(G\boxtimes H)=\{(u,v): u\in
V(G)\,\, and \,\, v\in V(H)\}$ and two vertices $(u_{1},v_{1})$ and
$(u_{2},v_{2})$ in $G\boxtimes H$ are adjacent if and only if
\begin{itemize}
  \item $u_{1}u_{2}\in E(G)$ and $v_{1}=v_{2}$ or
  \item $u_{1}=u_{2}$ and $v_{1}v_{2}\in E(H)$ or
  \item $u_{1}u_{2}\in E(G)$ and $v_{1}v_{2}\in E(H)$.
\end{itemize}
For a vertex $u\in V(G)$, the set of vertices $\{(u,v):v\in V(H)\}$
is called an $H$-layer and is denoted by $H^{u}$. Similarly, for a
vertex $v\in V(H)$, the set of vertices $\{(u,v):u\in V(G)\}$ is
called a $G-$layer and is denoted by $G^{v}$. Let $d_{G\boxtimes
H}((u_1,v_1),(u_2,v_2))$ denotes the distance between $(u_1,v_1)$ and
$(u_2,v_2)$. For $(u_1,v_1)(u_2,v_2)\in E(G\boxtimes H)$, the
local resolving neighborhood of edge $(u_1,v_1)(u_2,v_2)$ is denoted
by $L_{G\boxtimes H}((u_1,v_1)(u_2,v_2))$ and $L_G(u_1u_2)$
denotes the local resolving neighborhood of $u_1u_2\in E(G)$. The
following result gives the relationship between the distance of vertices in
$G\boxtimes H$ and the distance of vertices in graphs $G$ or $H$.
\begin{Remark}\label{strem}\cite{11}
Let $G$ and $H$ be two connected graphs. Then
$$d_{G\boxtimes H}((u_1,
v_1),(u_2, v_2)) = max\{d_G(u_1, u_2), d_H(v_1, v_2)\}.$$
\end{Remark}

\begin{Lemma}\label{sl}
Let $G$ and $H$ be two graphs of order $n_1\geq2$ and $n_2\geq2,$
respectively. Then
$$L_{G\boxtimes H}((u_i,v_j)(u_k,v_l))\subseteq\left\{
  \begin{array}{ll}
    V(G)\times L_H(v_jv_l), \,\,\ \mbox{if}\,\,\,i= k,\\
    L_G(u_iu_k)\times V(H),  \,\,\ \mbox{if}\,\,\,j=l,\\
    \{V(G)\times L_H(v_jv_l)\}\cup \{L_G(u_iu_k)\times V(H)\}\, otherwise.
  \end{array}
\right.$$
\end{Lemma}
\begin{proof}
Let $(u_i, v_j)(u_k, v_l)\in E(G\boxtimes H)$. If $i = k,$ then
$v_jv_l\in E(H)$. Let $(u_i,b)\in L_{G\boxtimes H}((u_i,v_j)(u_i,v_l))$, then
$d_{G\boxtimes H}((u_i, b), (u_i, v_j))\neq d_{G\boxtimes H}((u_i, b),(u_i, v_l))$. By Remark \ref{strem}, we have $d_H(b, v_j)\neq d_H(b,
v_l)$, therefore $b\in L_H(v_jv_l)$. Thus $(u_i, b)\in \{V(G)\times
L_H(v_jv_l)\}$. Analogously, if $j = l,$ then $u_iu_k\in E(G)$. Let $(a,v_j)\in L_{G\boxtimes H}((u_i,v_j)(u_k,v_j))$, then $d_{G\boxtimes H}((a, v_j), (u_i, v_j))\neq d_{G\boxtimes H}((a, v_j),(u_k, v_l))$. By Remark \ref{strem}, we have $d_G(a, u_i)\neq d_G(a, u_k)$, therefore $a\in L_G(u_iu_k)$. Thus $(a, v_j)\in \{L_G(u_iu_k)\times V(H)\}$. Finally, if $u_iu_k\in E(G)$ and $v_jv_l\in E(H)$, then two vertices $(u_i,v_j)$ and $(u_k,v_l)$ are locally resolved by either $(a,v_j)$ or $(u_i,b)$ or both. Let $(a,v_j)\in L_{G\boxtimes H}((u_i,v_j)(u_k,v_l))$, we
have
$$d_{G\boxtimes H}((u_i, v_j), (a, v_j)) = d_G(u_i, a)\neq d_G(u_k, a)$$
$$= max\{d_G(u_k, a), 1\} = d_{G\boxtimes H}((a, v_j), (u_k, v_l)).$$
Thus, $(a,v_j)\in \{L_G(u_iu_k)\times V(H)\}.$ Similar arguments hold for $(u_i,b)\in L_{G\boxtimes H}((u_i,v_j)(u_k,v_l))$.  Hence, $(a,vj),(u_i,b)\in \{V(G)\times L_H(v_jv_l)\}\cup \{L_G(u_iu_k)\times V(H)\}$ and we have the desired result.
\end{proof}
Now, we discuss some results involving the diameter or the radius of
$G$. For any two vertices $x$ and $y$ in a connected graph $G$, the
collection of all vertices which lie on an $x - y$ path of the shortest length is
known as the interval $I[x, y]$ between $x$ and $y$. Given a
non-negative integer $k$, we say that $G$ is adjacency $k-$resolved
if for every two adjacent vertices $x, y\in V(G)$, there exists
$w\in V(G)$ such that $d_G(y,w)\geq k$ and $x\in I[y,w],$ or
$d_G(x,w)\geq k$ and $y\in I[x,w]$. For example, path graphs and
cyclic graphs of order $n\geq2$ are adjacency $\lceil
\frac{n}{2}\rceil-$resolved.
\begin{Lemma}\label{s2}
Let $G$ be a non-trivial graph of diameter $diam(G)< k$ and let $H$
be an adjacency $k-$resolved graph of order $n_2$ and let
$(u_i,vj)(u_r,v_l)\in E(G\boxtimes H)$. Then
$$L_{G\boxtimes H}((u_i,v_j)(u_r,v_l))\subseteq \{L_G(u_iu_r)\times V(H)\}.$$
\end{Lemma}
\begin{proof}
Let $L_{G\boxtimes H}((u_i,v_j)(u_r, v_l))$ be the local resolving neighborhood of $(u_i,v_j)(u_r, v_l)\in E(G\boxtimes
H)$. We differentiate the following two cases.\\
\textbf{Case 1}: If $j = l$, then $u_iu_r\in E(G)$. Let $(u, v_j)\in L_{G\boxtimes H}((u_i,v_j)(u_r, v_j))$ then $d_{G.\boxtimes H}((u_i, v_j),(u, v_j))\neq d_{G\boxtimes H}((u_r,v_j),(u,v_j))$. By Remark \ref{strem}, we have $d_G(u_i, u)\neq
d_G(u_r, u)$, thus $u\in L_G(u_iu_r)$.\\
\textbf{Case 2}: If $v_j v_l\in E(H)$. Since $H$ is adjacency
$k-$resolved, there exists $v\in V(H)$ such that $(d_H(v, v_l)\geq
k$ and $v_j \in I[v, v_l])$ or $(d_H(v, v_j)\geq k$ and $v_l\in I[v,
v_j])$. Say $d_H(v, v_l)\geq k$ and $v_j\in I[v, v_l].$ In such a
case, as $diam(G) < k,$ for every $u\in L_G(u_iu_r)$ we have
$d_{G\boxtimes H}((u_i, v_j), (u, v)) = max\{d_G(u_i, u), d_H(v_j ,
v)\} < d_H(v, v_l) = max{d_G(u, u_r), d_H(v, v_l)} = d_{G\boxtimes
H}((u_r, v_l), (u, v)).$\\
Hence, $L_{G\boxtimes H}((u_i,v_j)(u_r,v_l))\subseteq
\{L_G(u_iu_k)\times V(H)\}.$
\end{proof}

\begin{Theorem}
Let $G$ be a non-trivial graph of diameter $diam(G)< k$ and let $H$
be an adjacency $k-$resolved graph of order $n_2$. Then
$$ldim_f(G\boxtimes H)\leq n_2.ldim_f(G)$$
\end{Theorem}
\begin{proof}
Let $(x,y)\in E(G\boxtimes H)$. Let $g : V (G)\rightarrow [0,1]$ be
a local resolving function of $G$ with $|g|=ldim_f(G)$.
We define a function $h:V(G\boxtimes H)\rightarrow[0,1],$
$$(x,y)\mapsto\left\{
 \begin{array}{ll}
    g(x), & \hbox{if $(x,y)\in G^{y}$,}\\
    0, & \hbox{otherwise.}
  \end{array}
  \right. $$
Note that $h$ is a local resolving function of $G\boxtimes
H$. Since $G$ has $n_2$ copies in $G\boxtimes H$, therefore $|h| \leq n_2.ldim_f(G)$. Hence, $ldim_f
(G\boxtimes H)\leq n_2.ldim_f (G)$.
\end{proof}
\begin{Theorem}\label{sthm}
Let $G$ and $H$ be two graphs of order $n_1\geq2$ and $n_2\geq2,$
respectively. Then
$$2\leq ldim_f(G\boxtimes H)\leq n_1.ldim_f(H)+n_2.ldim_f(G)-2ldim_f(G).ldim_f(H).$$
\end{Theorem}
\begin{proof}
Since $P_2\boxtimes P_2= K_4$ and $ldim_f(P_2\boxtimes P_2)=2$. So,
the lower bound follows. Let $(u,v)\in V(G\boxtimes H)$. Let $g_1 :
V (G)\rightarrow [0,1]$ be a local resolving function of $G$ with $|g_1|=ldim_f(G)$ and $g_2 : V(H)\rightarrow [0,1]$ be a local resolving function of $H$ with $|g_2|=ldim_f(H)$.
We define a function $h:V(G\boxtimes H)\rightarrow[0,1],$ with
$h(u,v)=g_1(u)+g_2(v)$. Note that $h$ is a local resolving function
of $G\boxtimes H$. Since $G$ has $n_2$ and $H$ has $n_1$ copies in $G\boxtimes H$, therefore $|h|=
n_1.ldim_f(H)+n_2.ldim_f(G)$.
Hence, $ldim_f(G\boxtimes H)\leq
n_1.ldim_f(H)+n_2.ldim_f(G)-2ldim_f(G).ldim_f(H)$.
\end{proof}
For the sharpness of upper bound in Theorem \ref{sthm}, let $G =K_n$ and $H= K_m$. Since $K_n\boxtimes K_m\cong K_{nm}$, therefore
$$ldim_f(K_n\boxtimes K_m)=\frac{nm}{2}\\
=n.ldim_f(K_m)+m.ldim_f(K_n)-2ldim_f(K_n).ldim_f(K_m).$$

 Now, we discuss general bounds for the fractional local metric
dimension of cartesian product of graphs. The cartesian product of
two graphs $G$ and $H$, denoted by $G\square H$, is a graph with the
vertex set $V(G\square H)=\{(u,v): u\in V(G)\,\, and \,\, v\in
V(H)\}$ and two vertices $(u_{1},v_{1})$ and $(u_{2},v_{2})$ in
$G\square H$ are adjacent if and only if
\begin{itemize}
 \item $u_{1}u_{2}\in E(G)$ and $v_{1}=v_{2}$ in $H$ or
 \item $u_{1}=u_{2}$ in $G$ and $v_{1}v_{2}\in E(H)$.
\end{itemize}
\begin{Remark}\label{carRem}\cite{11}
Let $G$ and $H$ be two connected graphs. Then
$$d_{G\square H}((u_1,
v_1),(u_2, v_2)) = d_G(u_1, u_2)+d_H(v_1, v_2).$$
\end{Remark}
\begin{Lemma}\label{lcar}
Let $G$ and $H$ be two graphs, then
$$L_{G\square H}((u_i,v_j)(u_k,v_l))=\left\{
  \begin{array}{ll}
    \bigcup\limits_{v\in L_H(v_jv_l)}\bigcup\limits_{u\in V(G)}\{uv\},   \,\,\,\     & \mbox{if}\,\,\,i= k,\\
    \bigcup\limits_{u\in L_G(u_iu_k)}\bigcup\limits_{v\in V(H)}\{uv\},   \,\,\,\     & \mbox{if}\,\,\,j=l.
  \end{array}
\right.$$
\end{Lemma}
\begin{proof}
For $(u_i, v_j)(u_k, v_l)\in E(G\square H)$ if $i = k,$ then
$v_jv_l\in E(H)$. Let $(u_i,v)\in L_{G\square  H}((u_i,v_j)(u_i,v_l))$, then
$d_{G\square H}((u_i, v), (u_i, v_j))\neq d_{G\square H}((u_i, v),(u_i, v_l))$. By Remark \ref{carRem}, we have $d_H(v, v_j)\neq d_H(v,
v_l)$, therefore $v\in L_H(v_jv_l)$. Thus $(u_i, v)\in \bigcup\limits_{v\in L_H(v_jv_l)}\bigcup\limits_{u\in V(G)}\{uv\} $. Now let $(u_i, v)\in \bigcup\limits_{v\in L_H(v_jv_l)}\bigcup\limits_{u\in V(G)}\{uv\}$, then $d_H(v,v_j)\neq d_H(v,v_l)$. By Remark \ref{carRem}, we have $d_{G\square H}((u_i, v), (u_i, v_j))\neq d_{G\square H}((u_i, v),(u_i, v_l))$. Thus $(u_i,v)\in L_{G\square  H}((u_i,v_j)(u_i,v_l))$.
Similar arguments hold for $j = l$. Hence, we have the desired result.

\end{proof}
\begin{Theorem}
Let $G$ and $H$ be two graphs. Then $ldim_f(G\square H)\geq
ldim_f(G).$
\end{Theorem}
\begin{proof}
Let $f$ be a local resolving function of $G\square H$ with $|f|=ldim_f(G\square H)$. We define a function $f_G:V(G)\rightarrow [0,1]$ such that $f_G(u)=\min\{1,\sum\limits_{v\in V(H)}f(u,v)\}$. For $u_1u_2\in E(G)$, we show that $f_G(L_G(u_1u_2))\geq 1$. If there exists an $x\in L_G(u_1u_2)$ with $f_G(x)=1$, then $f_G(L_G(u_1u_2))\geq 1$. Now, let for any $u\in V(G)$, $f_G(u)=\sum\limits_{v\in V(H)}f(u,v)$. Then
$$f_G(L_G(u_1u_2))= \sum\limits_{u\in L_G(u_1u_2)}\sum\limits_{v\in V(H)}f(u,v)$$
By Lemma \ref{lcar},
$$=f(L((u_1,v_0)(u_2,v_0))\geq 1,$$
Thus $f_G$ is a local resolving function of $G$. Since
$$|f_G|\leq \sum\limits_{u\in V(G)}\sum\limits_{v\in V(H)}f(u,v)=|f|,$$
hence $ldim_f(G\square H)\geq
ldim_f(G).$
\end{proof}
Since grid graph $P_n\square P_t$ is a bipartite graph and by Lemma \ref{b1}, we deduce\\
$ldim_f(P_n\square P_t)=1.$
\begin{Lemma}\label{um}
Let $G$ be a graph of order $n$, then $ldim_f (K_2\square G)\leq
ldim_f(G).$
\end{Lemma}
\begin{proof}
Let $V(K_2)=\{x, y\}$, $V(G) =\{u_1, u_2, . . . , u_n\}$ and $H =
K_2\square G.$ Then $V (H) = \{(x, u_i), (y, u_i) : i = 1, 2, . . .
, n\}.$ Let $f$ be a local resolving function of $G$ with $|f| =
ldim_f (G).$ Now we define $g : V (H) \rightarrow [0, 1]$ by $g((x,
u_i)) = g((y, u_i)) = \frac{f(u_i)}{2}, i =1, 2, . . . , n.$ We
claim that $g$ is a local resolving function for $H$. Let $uv\in
E(H)$, if $u=(x, u_i)$ and $v = (x, u_j ),$ then $\{\{x\}\times
L_G(u_i u_j)\}\subseteq L_H(u v)$ and hence $g(L_H(uv))\geq
f(L_G(u_iu_j))\geq 1.$ If $u = (x, u_i)$ and $v = (y, u_i),$ then
$L_H(uv) = V (H)$ and hence $g(L_H(uv))\geq 1.$ Thus, $g$ is a
local resolving function of $H$ with $|g| = |f|.$ Hence, $ldim_f (H)
\leq |f| = ldim_f (G).$
\end{proof}

\begin{Remark}
When $G$ is a bipartite graph and an odd cyclic graph, the bound given in Lemma \ref{um} is sharp. If $G$ is bipartite graph, then $ldim_f (K_2\square G) = 1 = ldim_f (G).$ If $n$ is an odd integer with $n\geq3$, then $ldim_f(K_2\square
C_n)=\frac{n}{n-1}.$
\end{Remark}
Let $G$ and $H$ be graphs with $V(H)=n$, Arumugam {\it et al.}
proved that the fractional metric dimension of $G\square H\geq
\frac{n}{2}$ if $dim_f(H)=\frac{n}{2}$ \cite{AM2}. Similar result
holds for the fractional local metric dimension with an alternative
proof as follows:
\begin{Theorem}\label{car1}
Let $G$ and $H$ be two connected graphs with order $m$, $n$
respectively and $ldim_f(H)=\frac{n}{2}$. Then $ldim_f(G\square
H)\geq \frac{n}{2}$.
\end{Theorem}
\begin{proof}
Since $ldim_f(H)=\frac{n}{2}$, by Theorem \ref{mthm}, every vertex
of $H$ has a true twin. Let $v$ has a true twin $w$ in $H$ then
$L_H(vw)=\{v,w\}$. By Lemma \ref{lcar}, it follows that $L_{G\square
H}((u,v)(u,w))=\{(x,v):x\in V(G)\}\cup\{(x,w):x\in V(G)\}$.\\
Now, let $f$ be a local resolving function of $G\square
H$. Then $f(L_{G\square H}((u,v)(u,w)))\geq 1$ for all $(u,v)(u,w)\in E(G\square H)$.
Hence $\sum\limits_{x\in V(G)}f((x,v))+\sum\limits_{x\in
V(G)}f((x,w))\geq 1$ for all $vw\in E(H)$. Adding these $n$
inequalities, we get
$$\sum\limits_{x\in V(H)}\sum\limits_{x\in
V(G)}f((x,v))+\sum\limits_{x\in V(G)}f((x,w))\geq n.$$
This implies $2|f|\geq n$. Hence $ldim_f(G\square H)\geq
\frac{n}{2}$.
\end{proof}
\begin{Corollary}\label{car2}
Let $G$ and $H$ be two connected graphs with order $m$, $n$
respectively and $ldim_f(G)=\frac{m}{2}$ and
$ldim_f(H)=\frac{n}{2}$. Then $ldim_f(G\square H)\geq
\max\{ldim_f(G),ldim_f(H)\}$.
\end{Corollary}
The bound given in Theorem \ref{car1} is sharp for $H=K_n$ as
follows:
\begin{Theorem}\label{carthm}
Let $G$ be any graph with $|V(G)|<n$, for all $n\geq 3$. Then $ldim_f(G\square
K_n)=\frac{n}{2}$.
\end{Theorem}
\begin{proof}
Let $|V(G)|=m$ with $m< n$. Let $V(G)=\{u_1,u_2,...,u_m\}$ and $V(K_n)=\{v_1,v_2,...,v_n\}$. Since by Theorem \ref{mthm}, $ldim_f(K_n)=\frac{n}{2}$, then by Theorem \ref{car1}, $ldim_f(G\square K_n)\geq \frac{n}{2}$. We claim that $|L_{G\square K_n}((u_i,v_r)(u_j,v_s))|\geq 2m$ for all $(u_i,v_r)(u_j,v_s)\in E(G\square K_n)$. For $(u_i,v_r)(u_j,v_s)\in E(G\square K_n)$, we have two cases. If $i=j$, then $r\neq s$ and by Lemma \ref{lcar}, we have $L_{G\square K_n}((u_i,v_r)(u_i,v_s))=\{(u_t,v_r): 1\leq t \leq m\}\cup \{(u_t,v_s): 1\leq t\leq m\}$. So $|L_{G\square K_n}((u_i,v_r)(u_i,v_s))|= 2m$. If $r=s$, then $i\neq j$ and by Lemma \ref{lcar}, we have $\{(u_i,v_t): 1\leq t \leq n\}\cup \{(u_j,v_t): 1\leq t\leq n\}\subseteq L_{G\square K_n}((u_i,v_r)(u_j,v_r))$. So $|L_{G\square K_n}((u_i,v_r)(u_j,v_r))|\geq 2n> 2m$.

Now the function $f:V(G\square K_n)\rightarrow [0,1]$ defined by $f((u,v))=\frac{1}{2m}$ for all $(u,v)\in V(G\square K_n)$ is a local resolving function of $G\square K_n$ with $|f|=\frac{|V(G\square K_n)|}{2m}=\frac{n}{2}$ and $ldim_f(G\square K_n)\leq \frac{n}{2} $. Hence, $ldim_f(G\square K_n)= \frac{n}{2}$.

\end{proof}
From Corollary \ref{car2}, we have the following result.
\begin{Theorem}
For $2\leq k\leq n$, $n\geq 3$, $ldim_f(K_k\square
K_n)=\frac{n}{2}$.
\end{Theorem}

\begin{proof}
The result follows from Theorem \ref{carthm}, when $k< n$. Consider the case when $k=n$. Since by Theorem \ref{mthm}, $ldim_f(K_n)=\frac{n}{2}$, then by Theorem \ref{car1}, $ldim_f(K_k\square K_n)\geq \frac{n}{2}$. Let $V(K_k)=\{u_1,u_2,...,u_k\}$ and $V(K_n)=\{v_1,v_2,...,v_n\}$. We claim that $|L_{K_k\square K_n}((u_i,v_r)(u_j,v_s))|\geq 2n$ for all $(u_i,v_r)(u_j,v_s)\in E(K_k\square K_n)$. For $(u_i,v_r)(u_j,v_s)\in E(G\square K_n)$, we have similar cases as in the proof of Theorem \ref{carthm} and we have $|L_{K_k\square K_n}((u_i,v_r)(u_j,v_s))|\geq 2n$.

Now the function $f:V(K_k\square K_n)\rightarrow [0,1]$ defined by $f((u,v))=\frac{1}{2n}$ for all $(u,v)\in V(K_k\square K_n)$ is a local resolving function of $k_k\square K_n$ with $|f|=\frac{n}{2}$ and $ldim_f(K_k\square K_n)\leq \frac{n}{2} $. Hence, $ldim_f(K_k\square K_n)= \frac{n}{2}$.
\end{proof}
\section{Summary and Conclusion}
In this paper, the concept of fractional local metric dimension of graphs
has been introduced. Graphs with $ldim_f(G)=\frac{|V(G)|}{2}$ have been
characterized. The fractional local metric dimension of some families of
graphs have been studied. Differences between the fractional metric dimension and the fractional local metric dimension of graphs have also been investigated. The fractional local metric dimension of strong and cartesian product of graphs have been studied and established some bounds on their fractional local metric dimension. However, it remains to determine the fractional local metric dimension of several other graph products.

\end{document}